\documentclass[10pt]{amsart}
\usepackage[utf8]{inputenc}
\usepackage{a4wide}
\usepackage[T1]{fontenc}

\usepackage{verbatim}
\usepackage{amsmath,amsthm,amssymb}
\usepackage[all]{xy}
\usepackage{xspace}
\usepackage{enumerate}
\usepackage{verbatim}

\newcommand{\R}{\ensuremath{\mathbb{R}}\xspace}
\newcommand{\C}{\ensuremath{\mathbb{C}}\xspace}

\newcommand{\id}{\vartriangleleft}

\newcommand{\ma}{\ensuremath{\mathcal{A}}\xspace}
\newcommand{\mb}{\ensuremath{\mathcal{B}}\xspace}
\newcommand{\me}{\ensuremath{\mathcal{E}}\xspace}

\newcommand{\pr}[2]{\ensuremath{\langle {#1},{#2}\rangle}}
\newcommand{\norm}[1]{\ensuremath{\|#1\|}}

\DeclareMathOperator{\gen}{span}

\theoremstyle{plain}
\newtheorem{thm}{Theorem}[section]
\newtheorem{prop}[thm]{Proposition}

\newtheorem{cor}[thm]{Corollary}          

\theoremstyle{definition}

\theoremstyle{remark} 
\newtheorem{rem}[thm]{Remark}

\title{On the amenability of partial and enveloping actions}
\author{Fernando Abadie} 
\address{CMAT, FC, Universidad de la
    Rep\'ublica. Igu\'a~4225, 11400, Montevideo,
    URUGUAY}
\email{fabadie@cmat.edu.uy}
\author{Laura Mart\'\i\ P\'erez} 
\address{CMAT, FC, Universidad de la
    Rep\'ublica. Igu\'a~4225, 11400, Montevideo,
    URUGUAY\\
    Department of Pure Mathematics, University of Waterloo,
    Waterloo, ON, CANADA, N2L 3G1}
\email{lau@cmat.edu.uy \\ lrmarti@uwaterloo.ca}

\begin{document}

\begin{abstract}
We prove that a partial action is amenable if and only if so is its
Morita enveloping action. As applications we prove that any partial
representation of a discrete group is positive definite, and we extend
a result of Zeller-Meier concerning the amenability
of discrete groups and the existence of invariant states
to partial actions.
\end{abstract}
\maketitle
\bigskip
\par In \cite{env}, the first-named author studied enveloping actions of 
partial actions on topological spaces and $C^*$-algebras, and compared
some of the structures naturally associated to actions, namely,
crossed products. However, most of the results obtained in that
work correspond to spatial constructions, that is, to reduced
crossed products. The goal of the present paper is to give versions
of some of the work in \cite{env}
for full crossed products, and then 
to combine both the spatial and the universal versions to the study of some   
amenability questions.    
\par Before proceeding we recall next some of the necessary
background on partial actions and their enveloping actions.
\par If $A$ is a $C^*$-algebra, a partial action of the group 
$G$ on $A$ is a pair $\alpha=(\{D_t\}_{t\in G},\{\alpha_t\}_{t\in
  G})$, where $D_t$ is a closed ideal of $A$
and $\alpha_t:D_{t^{-1}}\to
D_t$ is an isomorphism subject to the conditions: 
\begin{enumerate}
\item $\alpha_e=Id_A$, thus $D_e=A$, where $e$ is the identity of $G$, and   
\item $\alpha_{st}(a)=\alpha_s(\alpha_t(a))$ whenever $a\in
  D_{t^{-1}}$ is such that $\alpha_t(a)\in D_{s^{-1}}$ (so that   
  $a\in D_{t^{-1}s^{-1}}$).
\end{enumerate} 
For instance if $\beta:G\times B\to B$ is an action of $G$ on the 
$C^*$-algebra $B$ and $A\id B$, then the restriction $\beta|_A$ of 
$\beta$ to $A$ is a partial action of $G$ on $A$. More precisely,
$\beta|_A=(\{D_t\},\{\alpha_t\})$, where $D_t=A\cap\beta_t(A)$, and
$\alpha_t(a)=\beta_t(a)$, $\forall a\in D_{t^{-1}}$, $t\in G$. If
$\alpha$ is a partial action on $A$ for which there exists such an
action $\beta$ in a larger algebra $B$ containing $A$ as a closed
ideal and in addition 
$B=\overline{\gen}\{\beta_t(a):\, a\in A, t\in G\}$, then it is said
that $\beta$ is an enveloping action of~$\alpha$. The enveloping
action is unique up to isomorphism, if it exists, which is
not always the case. On the other hand, given a partial action
$\alpha$, there always exists another partial action $\alpha'$, which
is Morita equivalent to $\alpha$, such that $\alpha'$ does have an
enveloping action $\beta'$.
For this reason $\beta'$ is called a Morita
enveloping action of $\alpha$. Thus every partial action has a Morita
enveloping action that, in addition, is unique up to
Morita equivalence of partial actions. Moreover the reduced crossed
products of Morita equivalent partial actions are Morita equivalent,
as well as the reduced crossed products by a partial action and by its
Morita enveloping action. 
\par Suppose now that $\mb$ is a Fell bundle over the locally 
compact group $G$. Then the Banach *-algebra $L^1(\mb)$ has associated 
two distinguished $C^*$-completions. One of them is  
$C^*(\mb):=C^*(L^1(\mb))$ (the enveloping $C^*$-algebra of
$L^1(\mb)$), called the full cross-sectional $C^*$-algebra of $\mb$,
whose norm is the maximal $C^*$-norm on $L^1(\mb)$.  
The other one is $C^*_r(\mb):=\overline{\Lambda(L^1(\mb))}\subseteq   
\mathcal{L}(L^2(\mb))$, the so-called reduced cross-sectional algebra
of $\mb$, where $\Lambda$ is the left regular representation of
$L^1(\mb)$ on the full right $B_e$-Hilbert module $L^2(\mb)$, and   
$\mathcal{L}(L^2(\mb))$ denotes as usual the $C^*$-algebra of
adjointable operators on $L^2(\mb)$. The norm of $C^*_r(\mb)$ is
called the reduced norm. The regular representation extends to
$\Lambda:C^*(\mb)\to C^*_r(\mb)$, also called regular, that it is 
surjective. When this $\Lambda$ is an isomorphism, it is said that
$\mb$ is amenable.  Finally, if $\alpha$ is a partial action on 
the $C^*$-algebra $A$, it has an associated Fell
bundle~$\mb_\alpha$. Then $A\rtimes_\alpha G:=C^*(\mb_\alpha)$ and 
$A\rtimes_{\alpha,r} G:=C^*_r(\mb_\alpha)$ are called respectively the
(full) crossed product and the reduced crossed product of $A$ by
$\alpha$. The partial action is said to be amenable if $\mb_\alpha$ is
amenable. The reader is referred to \cite{env}, \cite{examen} , 
\cite{exeng} and \cite{fd} for more details on Fell bundles and their
amenability.     

\section{Amenability of partial actions}
\par Recall from \cite[Definition 4.2]{env} 
that a right ideal $\me=(E_t)_{t\in G}$ of a Fell bundle 
$\mb=(B_t)_{t\in G}$ is a sub-Banach bundle of $\mb$ such that
$\me\mb\subseteq\me$. The next theorem is the main result of the paper.    
\begin{thm}\label{thm:fb}
Let $\mb=(B_t)_{t\in G}$ be a Fell bundle over the locally compact
group $G$, $\ma=(A_t)_{t\in G}$ a sub-Fell bundle of $\mb$, and 
$\me=(E_t)_{t\in G}$ a right ideal of $\mb$ such that $\ma\subseteq\me$ 
and $\gen(\me^*\me\cap B_t)$ is dense in $B_t$, for all $t\in G$. Then 
the completion $C^*(\me)$ of $L^1(\me)$ in $C^*(\mb)$ is a bimodule
implementing a Morita equivalence between $C^*(\ma)$ and
$C^*(\mb)$. Moreover $\ma$ is amenable if and only if $\mb$ is
amenable.  
\end{thm}
\begin{proof}
From \cite[Theorem 3.2]{env} we have that $L^1(\ma)L^1(\me)\subseteq
L^1(\me)$, $L^1(\me)L^1(\me)^*=L^1(\ma)$, and that $L^1(\me)$ is a
right $L^1(\mb)$-module such that
$\overline{\gen}\,L^1(\me)^*L^1(\me)=L^1(\mb)$. Therefore $L^1(\me)$
is a $(L^1(\ma)-L^1(\mb))$-bimodule. Let $\norm{\ }_\ma$ and $\norm{\
}_\mb$ be the maximal $C^*$-norms 
on $L^1(\ma)$ and $L^1(\mb)$
respectively. We clearly have that $\norm{f}_\ma\geq\norm{f
}_\mb$, $\forall f\in L^1(\ma)$. For $\xi\in L^1(\me)$ define
$\norm{\xi}^2_E:=\norm{\xi*\xi^*}_\ma$, and
$\norm{\xi}^2_F:=\norm{\xi^**\xi}_\mb$. Then $\norm{\xi}^2_F 
=\norm{\xi^**\xi}_\mb=\norm{\xi*\xi^*}_\mb\leq\norm{\xi*\xi^*}_\ma
=\norm{\xi}^2_E$. Also let $E$ and $F$ be the 
completions of $L^1(\me)$ with respect to $\norm{\ }_E$ and $\norm{\
}_F$ respectively. Then the module structures of  
$L^1(\me)$ extend to $E$ and $F$, so that $E$ is a left full
$C^*(\ma)$-Hilbert module, and $F$ is a right full $C^*(\mb)$-Hilbert 
module, where $\pr{\xi}{\eta}_E=\xi*\eta^*$ and
$\pr{\xi}{\eta}_F=\xi^**\eta$ for $\xi,\eta\in L^1(\me)$ (note that
$F=C^*(\me)$). To prove the first assertion of the theorem, by
\cite{z} it is enough to show that $\norm{\xi}_E=\norm{\xi}_F$, $\forall
\xi\in L^1(\me)$. To this purpose consider the set
$J:=\gen\{\xi^**\eta:\,\xi,\eta\in L^1(\me)\}$, which is a dense
*-ideal of $L^1(\mb)$. Given $g\in L^1(\mb)$, denote by
$C_g:L^1(\me)\to L^1(\me)$ the convolution operator such that
$C_g(\zeta)=\zeta*g$, $\forall \zeta\in\me$. Recall that if $x,y\in
E$, then $\theta_{x,y}:E\to E$ is defined as
$\theta_{x,y}(z)=\pr{z}{x}_Ey$, and the $C^*$-algebra
$\mathcal{K}_{C^*(\ma)}$ of compact operators on the
$C^*(\ma)$-Hilbert module $E$ is the completion of
$\gen\{\theta_{x,y}:\, x,y\in E\}$ with respect to the operator
norm. Note now that if $\xi_1,\ldots,\xi_n,\eta_1,\ldots,\eta_n\in
L^1(\me)$ and $g=\sum_{j=1}^n\xi_j^**\eta_j\in J$, then 
$C_g(\zeta)=\sum_{j=1}^n(\zeta*\xi_j^*)*\eta_j 
=\sum_{j=1}^n\theta_{\xi_j,\eta_j}(\zeta)$. That is: $C_g$ is
precisely the restriction of $\sum_{j=1}^n\theta_{\xi_j,\eta_j}$ to
$L^1(\me)$, or, in other words, $C_g$ is $\norm{\ }_E$-bounded, and
its continuous extension to $E$ is precisely the compact operator 
$\sum_{j=1}^n\theta_{\xi_j,\eta_j}$. Let $\rho':J\to
\mathcal{K}_{C^*(\ma)}$ be the map such that $\rho'(g)$ 
is the
continuous extension of $C_g$ to all of $E$. This is a homomorphism of
*-algebras. Since $J$ is a *-ideal of the Banach algebra
$L^1(\mb)$ and $\mathcal{K}_{C^*(\ma)}$ is a $C^*$-algebra, then by
\cite[XI-19.11]{fd} $\rho'$ extends uniquely to a homomorphism of
Banach *-algebras $\rho:L^1(\mb)\to \mathcal{K}_{C^*(\ma)}$. This
implies that $\norm{\rho(g)}\leq\norm{g}_\mb$, $\forall g\in
L^1(\mb)$. On the other hand, for $\xi\in L^1(\me)$ we have 
$\rho'(\xi^**\xi)=\theta_{\xi,\xi}$. Hence: 
\[\norm{\xi}_E^2=\norm{\theta_{\xi,\xi}}=\norm{\rho'(\xi^**\xi)}
\leq\norm{\xi^**\xi}_\mb=\norm{\xi}_F^2\leq\norm{\xi}_E^2,\]  
where the last inequality was shown after the definition of both
norms. This shows that $\norm{\ }_E=\norm{\ }_F$, which ends the proof
of the  first part of the theorem. 
\par By \cite[Corollary~3.2]{env} we have that the amenability of
$\ma$ implies that of $\mb$. So to prove the second assertion of the
theorem suppose
$\mb$ is amenable. We have just seen that the maximal norm on
$L^1(\ma)$ is the restriction to this algebra of the maximal norm of
$L^1(\mb)$, a fact that was already known for the reduced norms
(\cite[Corollary~3.1]{env}). So if the maximal and the reduced norms
of $L^1(\mb)$ agree, their restrictions to $L^1(\ma)$ also
coincide, that is: $\ma$ is amenable whenever $\mb$ is amenable.   
\end{proof}
\begin{rem} A more precise result can be given: there are isomorphisms
  of partially ordered sets between the families of $C^*$-seminorms on
  $L^1(\ma)$, $L^1(\me)$ and $L^1(\mb)$, such that for corresponding
  $C^*$-seminorms the respective completions form a Morita equivalence
  system. See \cite{trings} for details. 
\end{rem}
\par The theorem above allows us to adapt in a straightforward way the 
proofs of Theorem~3.3, Proposition~4.5 and Proposition~4.6 of
\cite{env} to the case of full crossed products, an easy
task that we leave to the reader. Then the following corollaries are 
obtained:   
\begin{cor}\label{cor:fullcp}
If the partial action $\alpha$ of the locally compact group $G$ on the
$C^*$-algebra $A$ has a Morita enveloping action $\beta$ acting on the 
$C^*$-algebra $B$, then $A\rtimes_\alpha G$ and $B\rtimes_\beta G$ are
Morita equivalent. Moreover $\alpha$ is amenable if and only if
$\beta$ is amenable.     
\end{cor}
\begin{cor}\label{cor:meqpa}
If $\alpha_1$ and $\alpha_2$ are Morita equivalent partial actions,
then $A\rtimes_{\alpha_1} G$ and $B\rtimes_{\alpha_2} G$ are Morita
equivalent. Moreover $\alpha_1$ is amenable if and only if $\alpha_2$
is amenable. 
\end{cor}

\section{Invariant states}
\par Let $\alpha$ be a partial action of a group $G$ on a
$C^*$-algebra $A$. A map $\phi$ on $A$  is called
$\alpha$-invariant whenever $\phi(\alpha_t(a))=\phi(a)$, $\forall t\in
G$ and $a\in D_{t^{-1}}$. We will denote by $A_\alpha'$ the Banach
space of $\alpha$-invariant bounded linear functionals on $A$. The
space $A_\alpha'$ is partially ordered: if $\phi_1,\phi_2\in
A'_\alpha$, then $\phi_1\geq\phi_2$ if $\phi_1-\phi_2$ is a positive
linear functional.      
\par In this section we show that if the partial action $\alpha$ on
the $C^*$-algebra $A$ has enveloping action $\beta$, acting on a
unital $C^*$-algebra $B$, then the map $\psi\mapsto
\frac{\psi|_A}{\norm{\psi|_A}}$ is a 
bijection between the set of $\beta$-invariant states of
$B$ onto the set of $\alpha$-invariant states of $A$. The requirement
that $B$ is unital is necessary.  
\begin{prop}\label{prop:inv}
Let $\alpha=(\{D_t\},\{\alpha_t\})$ be a partial action of the
locally compact group $G$ on the 
$C^*$-algebra $A$, and suppose that $\alpha$ has an enveloping action
$\beta$ acting on a $C^*$-algebra $B$. Then the map $R_A:B_\beta'\to
A_\alpha'$ given by $\psi\mapsto\psi|_A$ is an injective bounded
linear map that preserves the order. If $B$ is unital, then $R_A$ is an
isomorphism of Banach spaces and ordered sets. Therefore, if $B$ is
unital, there exists a $\beta$-invariant state of $B$ if and only if 
there exists an $\alpha$-invariant state of $A$. 
\end{prop}
\begin{proof}
It is clear that $\psi\mapsto\psi|_A$ is linear and
contractive. Suppose that $\psi\in B_\beta'$, and let
$\phi=\psi|_A$. Then $\phi$ is a nonzero functional: otherwise
we would have, for every $b=\sum_{j=1}^n\beta_{t_j}(a_j)$, $a_j\in A$
$\forall j$:
$\psi(b)=\sum_{j=1}^n\psi(\beta_{t_j}(a_j))=\sum_{j=1}^n\psi(a_j) 
=\sum_{j=1}^n\phi(a_j)=0$, which implies $\psi=0$. The functional
$\phi$ is $\alpha$-invariant: if $a\in D_{t^{-1}}$, then
$\phi(\alpha_t(a))=\psi(\beta_t(a))=\psi(a)=\phi(a)$. It is clear that
$\phi$ is positive if $\psi$ is positive. This shows that
$R_A:B_\beta'\to A_\alpha'$ is an injective contractive linear map
that preserves the order.   
\par Suppose conversely that $\phi$ is a $\alpha$-invariant bounded 
linear functional on $A$, and that $B$ is unital. Then
$\phi_t:=\phi\beta_{t}^{-1}$ is a bounded linear functional  
on $\beta_t(A)$. Since $B$ is unital and the sum of the
ideals $\beta_t(A)$ is a dense ideal in $B$, there exist elements 
$t_1,\ldots,t_k$ in $G$ such that $B=A_1+\cdots+A_k$, where
$A_j=\beta_{t_j}(A)$, $\forall 
j=1,\ldots,k$. Let $u_j\in A_j$ be such that $1=u_1+\cdots+u_k$, and
let $v_j=\beta_{t_j}^{-1}(u_j)\in A$, $\forall j=1,\ldots,k$. Define
$\psi:B\to\C$ by $\psi(b)=\sum_{j=1}^k\phi_{t_j}(bu_j)$, $\forall b\in
B$, which is clearly a bounded linear functional on $B$. Note that if
$a\in A$ we have $au_j\in 
A\cap\beta_{t_j}(A)=D_{t_j}$. Since $\phi$ is
$\alpha$-invariant, it follows that
$\psi(a)=\sum_{j=1}^k\phi_{t_j}(au_j)
=\sum_{j=1}^k\phi\beta_{t_j^{-1}}(au_j)
=\sum_{j=1}^k\phi(\alpha_{t_j^{-1}}(au_j))
=\phi(\sum_{j=1}^kau_j)=\phi(a)$. Therefore $\psi|_{A}=\phi$. Now for
$a\in A$ and $t\in G$, we have
$\beta_t(a)u_j\in A_j$ and
$\phi_j(\beta_t(a)u_j)=\phi(\beta_{t_j^{-1}t}(a)v_j)$. Since
$\beta_{t_j^{-1}t}(a)v_j\in D_{t_j^{-1}t}$ we have
$\phi_j(\beta_t(a)u_j)
=\phi(\alpha_{t_j^{-1}t}\alpha_{t^{-1}t_j}(\beta_{t_j^{-1}t}(a)v_j))$.
Then $\phi_j(\beta_t(a)u_j)
=\phi(\alpha_{t^{-1}t_j}(\beta_{t_j^{-1}t}(a)v_j))$, because $\phi$ is
$\alpha$-invariant. Thus 
\[\phi_j(\beta_t(a)u_j)=\phi(\alpha_{t^{-1}t_j}(\beta_{t_j^{-1}t}(a)v_j))
=\phi(\beta_{t^{-1}t_j}(\beta_{t_j^{-1}t}(a)v_j))
=\phi(a\beta_{t^{-1}}(u_j)).\] Hence, since 
$\sum_{j=1}^k\beta_{t^{-1}}(u_j)=1$ we get       
\[\psi(\beta_t(a))
=\sum_{j=1}^k\phi_j(\beta_t(a)u_j)
=\phi(a\sum_{j=1}^k\beta_{t^{-1}}(u_j))
=\phi(a)
=\psi(a).\]    
The $\beta$-invariance of $\psi$ follows now by its
linearity. This shows that $R_A$ is also surjective, thus an
isomorphism of Banach spaces. Moreover if $b\in B^+$, then 
$b=\beta_{t_1}(a_1)+\cdots+\beta_{t_k}(a_k)$, with $a_j\in A^+$, from 
which we have that
$\psi(b)=\sum_{j=1}^k\psi(\beta_{t_j}(a_j))=\sum_{j=1}^k\phi(a_j)\geq
0$. Therefore $R_A$ is also an isomorphism of partially ordered
sets. Finally, from the above it follows that
$\psi\mapsto\frac{\psi}{\norm{\psi}}$ is a bijection from the set of
$\beta$-invariant states of $B$ onto the set of $\alpha$-invariant
states of $A$. 
\end{proof}
\par When $B$ is not unital, it is not true that every positive
$\alpha$-invariant linear functional on $A$ can be extended to a
$\beta$-invariant linear functional on $B$. For instance, consider
$B=C_0(\R)$ and $\beta :\R\times B\to B$ such that
$\beta_x(b)(t):=b(t-x)$. Let $A=C_0(0,1)$ and
$\alpha:=\beta|_{A}$. 
It is clear that $\beta$ is the enveloping
action of~$\alpha$. Now the functional $\phi:A\to\C$ such that
$\phi(a)=\int_{[0,1]}a\,dm$, where $m$ is Lebesgue measure, is
$\alpha$-invariant. On the other hand, Lebesgue measure on $\R$ is
translation invariant, so if $\phi$ could be extended to a
$\beta$-invariant linear functional on $B$, this functional on the
dense subalgebra $C_c(\R)$ should be integration with respect to
Lebesgue measure, which is not a bounded linear map. However, as shown
in \cite{lau}, in the
commutative case one can see that any probability measure on the
spectrum $\hat{A}$ of $A$ can be uniquely extended to a Radon measure
on the spectrum $\hat{B}$ of $B$.     

\section{Partial actions of discrete groups} 
\par In this final small section we give two applications of the
preceding results. 
\begin{thm}\label{thm:partialreps}
Let $u:G\to B(H)$ be a partial representation of a discrete group
$G$. Then $u$ is a positive definite map.   
\end{thm}
\begin{proof}
Just observe that in view of Corollary~\ref{cor:fullcp} above it is
possible now to remove the hypothesis of amenability of the group $G$
that was necessary in \cite[Proposition~3.3]{env}. 
\end{proof}
\par A Fell bundle over an amenable group $G$ is necessarily amenable
(\cite{examen}, \cite{exeng}). As a partial converse of this fact, for
Fell bundles associated to partial actions of discrete 
groups we have the following result: 
\begin{thm}\label{thm:zm}
Let $\alpha$ be a partial action of the discrete group $G$ on the
$C^*$-algebra $A$, and suppose that $\alpha$ has enveloping action
$\beta$ acting on a unital $C^*$-algebra $B$. Then the following are
equivalent: 
\begin{enumerate}
\item $G$ is amenable. 
\item $\beta$ is amenable and there exists a $\beta$-invariant state
      in $B$.
\item $\alpha$ is amenable and there exists an $\alpha$-invariant
      state in $A$.  
\end{enumerate}
\end{thm}
\begin{proof}
The equivalence between the first two assertions is
\cite[5.2]{zm}. Combining Corollary~\ref{cor:fullcp}
and Proposition~\ref{prop:inv} we obtain the equivalence between the
last two assertions. 
\end{proof}

\end{document}